\documentclass[11pt]{article}
\usepackage{dsfont}
\usepackage{amsmath, amsthm, amssymb, mathrsfs}
\usepackage[abbrev,non-sorted-cites]{amsrefs}
\usepackage[all]{xy}
\usepackage{CJK}
\usepackage{grffile}
\usepackage{graphicx}
\newtheorem{thm}{Theorem}[section]
\newtheorem{cor}[thm]{Corollary}
\newtheorem{lem}[thm]{Lemma}
\newtheorem{ex}[thm]{Example}
\newtheorem{definition}[thm]{Definition}
\newtheorem{rmk}[thm]{Remark}

\newtheorem{prop}[thm]{Proposition}
\newtheorem{conj}[thm]{Conjecture}

\begin{document}

\title{Reduced Open Gromov-Witten Invariants on HyperK\"ahler Manifolds}
\author{Yu-Shen Lin}

\maketitle
\section{Introduction}

   Inspired by closed string theory and pioneer work of Gromov \cite{G}, Gromov-Witten theory now becomes an important technique to study symplectic manifolds. For instance, Floer conquered the celebrated Arnold conjecture using the technology of pseudo-holomorphic curves. Naively, Gromov-Witten invariants count the number of (pseudo-)holomorphic curves in a symplectic manifold with prescribed incident conditions. However, for hyperK\"ahler manifolds, the standard tangent-obstruction theory of the moduli space of curve has a trivial quotient. Therefore, the genus zero Gromov-Witten invariants vanishes for any hyperK\"ahler manifolds. The vanishing of the invariants reflects the fact that hyperK\"ahler manifolds usually admit deformation with no holomorphic curves and Gromov-Witten invariants are deformation invariants. To retrieve non-trivial invariants, people developed the reduced Gromov-Witten invariants by removing the trivial factor \cite{L}\cite{L5}\cite{MP}. 
    
   In symplectic geometry, holomorphic discs are broadly used to understand the geometry of Lagrangians. In the context of mirror symmetry, holomorphic discs on a Calabi-Yau manifold play important roles of quantum correction of complex structure of the mirror Calabi-Yau. However, it is hard to defined open Gromov-Witten type invariants in general due to the existence of real codimension one boundaries of the moduli spaces of pseudo-holomorphic discs, which makes the integration or intersection theory on those moduli spaces ambiguous. The typical way to overcome this defect is to add some decoration to the relevant moduli spaces such as anti-symplectic involution \cite{S2} or $S^1$-action \cite{L2}. 
    
   Here in this paper, we are going to define a version of reduced open Gromov-Witten invariants on hyperK\"ahler manifolds which can be viewed as an hybrid. We start with a hyperK\"ahler manifold $X$ with holomorphic Lagrangian fibration. There exists a natural $S^1$-family of complex structures in the twistor line making holomorphic Lagrangian fibration become a special Lagrangian fibration. In this setting, changing the boundary condition is similar to changing stability conditions. The "invariant" we defined here will only be constant locally but might jump when the torus fibre moves through a real codimension one subset on the base, which we will call them the wall of marginal stability. The jump of invariants satisfy the Kontsevich-Soibelman wall-crossing formula \cite{KS1} in the examples we studied.

Despite the interest in symplectic geometry, the reduced open
Gromov-Witten invariant is also related to twistorial construction
of hyperK\"ahler metric constructed in \cite{GMN} and mirror symmetry. 

This paper is organized as follows. In section 2, we review the definition of hyperK\"ahler manifolds and the hyperK\"ahler rotation trick. In section 3, we define the reduced Gromov-Witten invariants for hyperK\"ahler manifolds. We will focus on the example when $X$ is an elliptic K3 surfaces and see the wall-crossing phenomenon in section 4. As applications of this new open Gromov-Witten invariants, we will discuss the tropical geometry on K3 surfaces in section 5. We will talk about the relation between the twistorial construction of hyperK\"ahler metrics \cite{GMN} in section 6.

\section*{Acknowledgements} The author would like to thank his thesis advisor Shing-Tung Yau for constant support and encouragement. This paper is include part of the author's thesis. The author would also like to thank the organizers of ICCM 2013 and National Taiwan University for invitation and hospitality. 

\section{HyperK\"ahler Manifolds and HyperK\"ahler Rotation Trick}
\begin{definition}
 A complex manifold $X$ of dimension $2n$ is called a hyperK\"ahler manifold if its holonomy group falls in $Sp(n)$.
\end{definition}

\begin{ex}
  Any compact complex K\"ahler manifold admits a holomorphic symplectic $2$-form is hyperK\"ahler \cite{Y1}. In particular, K3 surfaces are hyperK\"ahler.
\end{ex}

\begin{ex} There are also non-compact examples of hyperK\"ahler manifolds such as cotangent bundles of K\"ahler manifolds, Hithcin moduli spaces and the Ooguri-Vafa space (see section \ref{801}).
\end{ex}

Let $X$ be a hyperK\"ahler manifold and $g$ be the corresponding
hyperK\"ahler metric from the definition. Then there exists integrable complex structures $J_1,J_2,J_3$ satisfying quaternion relation. 
  \begin{align*}
    \omega(\cdot,\cdot)=g(J_3\cdot,\cdot) 
   \end{align*} is a K\"ahler form and 
  \begin{align*}
   \Omega(\cdot,\cdot)=g(J_1\cdot,\cdot)+ig(J_2\cdot,\cdot)
  \end{align*} is a holomorphic $2$-form with respect to the complex structure $J_3$. Moreover, $\underline{X}$ admits a family of
complex structures parametrized by $\mathbb{P}^1$, called twistor line. Explicitly, they are given
by
\begin{equation*}
   J_{\zeta}=\frac{i(-\zeta+\bar{\zeta})J_1-(\zeta+\bar{\zeta})J_2+(1-|\zeta|^2)J_3}{1+|\zeta|^2},
   \hspace{3mm}
   \zeta\in \mathbb{P}^1.
\end{equation*}
 Moreover, the holomorphic
symplectic $2$-forms $\Omega_{\zeta}$ with respect to the compatible
complex structure $J_{\zeta}$ are given by
\begin{equation}  \label{38}
  \Omega_{\zeta}=-\frac{i}{2\zeta}\Omega+\omega-\frac{i}{2}\zeta\bar{\Omega}.
  \end{equation} In particular, straightforward computation from (\ref{38}) gives
\begin{prop} \label{49}
Assume $\zeta=e^{i\vartheta}$, then we have
  \begin{align*}
     \omega_{\vartheta}:=\omega_{\zeta}&=-\mbox{Im}(e^{-i\vartheta}\Omega),\\
     \Omega_{\vartheta}:=\Omega_\zeta&=\omega-iRe(e^{-i\vartheta}\Omega).
  \end{align*}
\end{prop}

\begin{rmk} \label{300}
   Let $L$ be a holomorphic Lagrangian in $(\underline{X},\omega,\Omega)$, namely,
   $\Omega|_L=0$. Assume that the
north and south pole of the twistor line are given by
$(\omega,\Omega)$ and $(-\omega,\bar{\Omega})$ respectively, making
$L$ a holomorphic Lagrangian. The hyperK\"ahler structures
corresponding to the equator $\{\zeta=e^{i\vartheta}:|\zeta|=1\}$
make $L$ a special Lagrangian in
$X_{\vartheta}=(\underline{X},\omega_{\vartheta},\Omega_{\vartheta})$,i.e. $\omega_{\vartheta}|_{L}=\mbox{Im}\Omega_{\vartheta}|_{L}=0$ by Proposition \ref{49}.
In particular, if $(\underline{X},\omega,\Omega)$ admits a holomorphic
Lagrangian fibration, then it induces a special
Lagrangian fibrations on $X_{\vartheta}$ for each $\vartheta\in S^1$.
This is the so-called hyperK\"ahler rotation trick.
\end{rmk}

\section{Reduced Open Gromov-Witten Invariants on HyperK\"ahler
Manifolds}

For simplicity, we will assume the hyperK\"ahler manifold $(X,\omega,\Omega)$ (not necessarily compact) admits an abelian fibration structure $X\rightarrow B$ such that the generic fibres are complex torus.  
Let $\Delta\subseteq B$ be the discriminant locus (also referred as singularity of the affine structures later on) of the fibration and $B_0=B\backslash \Delta$. We will denote the fibre over $u\in B$ by $L_u$.

Consider the following exact sequence of local systems 
\begin{align}\label{731}
     \bigcup_{u\in B_0}H_2(X) \rightarrow \Gamma=\bigcup_{u\in B_0} H_2(X, L_u) \rightarrow \Gamma_g= \bigcup_{u\in B_0} H_1(L_u) \rightarrow 0,
\end{align}

 The topological pairing on $\Gamma_g$ is a non-degenerate symplectic pairing which lift to a degenerate skew-symmetric pairing on $\Gamma$ with kernel $\bigcup_{u\in B_0}H_2(X)$. We define the central charge $Z$ to be the period 
\begin{align*}
 Z:\Gamma &\hspace{2mm} \longrightarrow \hspace{5mm}\mathbb{C}\\
\gamma_u &\longmapsto Z_{\gamma_u}(u)= \int_{\gamma_u}\Omega, 
\end{align*}
for each $\gamma_u\in H_2(X,L_u)$. The integral is well-defined because
$\Omega|_L=0$. We will call $|Z_{\gamma}|$ the energy of the relative class $\gamma$. The following lemma is straight forward computation:

\begin{lem} \label{34}
    For any $v \in TB_0$, we have
   \begin{equation}\label{735}
   dZ_{\gamma}(v)=\int_{\partial \gamma}\iota_{\tilde{v}}\Omega,
   \end{equation}
   where $\tilde{v} \in TX$ is any lifting of $v$.
\end{lem}
\begin{proof}It is straightforward to check that the right hand side of (\ref{735}) is independent of choices of lifting $\tilde{v}$. Since $\Omega|_L=0$, we view $\Omega$ as the
   element $(\Omega,0)\in H^2(X,L)$. From the variational formula of
   relative pairing,
   \begin{align*}
      dZ_{\gamma}(v)&=\mathcal{L}_v \langle \gamma,(\Omega,0)\rangle \\
                    &=\langle \gamma, (\iota_{\tilde{v}}d\Omega,
                    \iota_{\tilde{v}}(0-\Omega))\rangle=\int_{\partial
                    \gamma}\iota_{\tilde{v}}\Omega.
   \end{align*}
\end{proof}

Given a point $u_0\in B_0$ and an element $\gamma_{u_0}\footnote{we might drop the subindex $u_0$ when there is no ambiguity later on}\in
\Gamma_{u_0}$, there exists a neighborhood $\mathcal{U}$ of $u_0$ in $B_0$
and a neighborhood of $\tilde{\mathcal{U}}$ of $\gamma_{u_0}$ in $\Gamma$.  such
that $\mathcal{U}$ is homeomorphic to $\tilde{\mathcal{U}}$ via the projection. Under this identification, we can define a complex structure on $\Gamma$ and have the following important observation.
\begin{cor}
  The central charge $Z:\Gamma\rightarrow \mathbb{C}$ is a
  holomorphic function.
\end{cor}
\begin{proof}
  Since any $(0,1)$-vector on $TB_0$ can be expressed in term of
  $v+iJv$ for some $v\in TB_0$, where $J$ is the almost complex structure. We have 
     \begin{align*}
        (v+iJv)Z=\int_{\partial
        \gamma}\iota_{(\tilde{v}+iJ\tilde{v})}\Omega=0.
     \end{align*}The latter equality holds because $\Omega$ is a
     $(2,0)$-form and $\tilde{v}+iJ\tilde{v}$ is always a $(0,1)$-vector.
     Notice that for $y$ near a singularity of the affine structure,
     $\gamma$ represents the relative class of Lefschetz thimble,
     then $Z_{\gamma}$ is bounded in a neighborhood of the
     singularity and thus is a removable singularity.
\end{proof}

\begin{cor} \label{350}
   Let $\gamma\in \Gamma$, then $dZ_{\gamma}\neq 0$ whenever $Z_{\gamma}$ is defined. 
\end{cor}
\begin{proof}
  The corollary follows immediately from Lemma \ref{34} and $L$ is always Lagrangian with respect to the symplectic $2$-form $\mbox{Re}(e^{-i\vartheta}\Omega)$. 
\end{proof}

Recall that the hyperK\"ahler triple $(\omega,\Omega)$
will induce a $\mathbb{P}^1$-family of hyperK\"ahler structures. Let $L$ be a holomorphic
Lagrangian fibre with respect to $\Omega$, then by Remark \ref{300} there is an $S^1$-family of hyperK\"ahler structures in the twistor family such that $L$ is a special
Lagrangian. Let $\mathcal{M}_{k,\gamma}(\mathfrak{X},L)$ be the moduli space of stable maps of pseudo-holomorphic discs in the above $S^1$-family in the relative class $\gamma\in H_2(X,L)$ with boundaries on the fixed
special Lagrangian $L$ and with $k$ boundary marked points in
counter-clockwise order. Let $ev_i$ be the evaluation map of the $i$-th marked point. 
The virtual dimension of $\mathcal{M}_{k,\gamma}(\mathfrak{X},L)$ is given by $n+k-2$

 Assume $\gamma_u \in H_2(\underline{X},L)$ can be realized as the image of a holomorphic map in $X_{\vartheta}$, then 
  \begin{align*}
     Z_{\gamma}(u)=e^{i\vartheta}\int_{\gamma}e^{-i\vartheta}\Omega=e^{i\vartheta}i \int_{\gamma}\mbox{Im}(e^{-i\vartheta}\Omega).
  \end{align*} The last identity is because $\int_{\gamma}\Omega_{\vartheta}=0$.
In particular, we have $\mbox{Arg}Z_{\gamma}(u)=\vartheta+\pi/2$. In other words, the phase of central charge $\mbox{Arg}Z_{\gamma}$ a priori determines the obstruction for the ob which complex structure on the equator can support $\gamma$ as a holomorphic cycle. In this case, the central charge is nothing but
    \begin{align}\label{736}
     Z_{\gamma}= \mbox{(symplectic area) }e^{i(\vartheta+\pi/2)}
    \end{align} from the Proposition \ref{49}.
In particular, the moduli space for family
$\mathcal{M}_{k,\gamma}(\mathfrak{X},L)$ has the same underlying
space as the usual moduli space of holomorphic discs
$\mathcal{M}_{k,\gamma}(X_{\vartheta},L)$, where  $\vartheta=\mbox{Arg}Z_{\gamma}-\pi/2$. We might drop the subindex $\vartheta$ when the target is clear. However, the tangent-obstruction theory (or the Kuranishi structures) on $\mathcal{M}_{k,\gamma}(\mathfrak{X},L)$ and $\mathcal{M}_{k,\gamma}(X_{\vartheta},L)$ are different. 

Now assume that the moduli space $\mathcal{M}_{\gamma_u}(\mathfrak{X},L_u)$ has  non-empty real codimension one boundary for some $\gamma_u\in H_2(X,L_u)$. Namely, we have 
  \begin{align*}
  \gamma_u=\gamma_{1,u}+\gamma_{2,u} \in H_2(X,L_u)
  \end{align*} and $\mathcal{M}_{\gamma_{i,u}}(X_{\vartheta},L_u)$ are non-empty, for $i=1,2$ and $\vartheta\in S^1$. In particular, we have $Z_{\gamma_1}(u)Z_{\gamma_2}(u)\neq 0$ from (\ref{736}) and 
     \begin{align*}
      \mbox{Arg}Z_{\gamma}(u)=\mbox{Arg}Z_{\gamma_1}(u)=\mbox{Arg}_{\gamma_2}(u)=\vartheta+\pi/2.
     \end{align*}
     The interesting implication is that we may not always have bubbling phenomenon of the moduli space $\mathcal{M}_{\gamma}(\mathfrak{X},L_u)$ unless the torus fibre $L_u$ sits over the locus characterized by
   \begin{align} \label{700}
    \mbox{Arg}Z_{\gamma_1}=\mbox{Arg}Z_{\gamma_2}.
   \end{align} 
Assume that $Z_{\gamma_1}$ is not a multiple of $Z_{\gamma_2}$. Since the central charges are holomorphic functions, the equation (\ref{700}) locally is pluriharmonic and defines a real analytic pseudoconvex hypersurface. In particular, the mean value property of pluriharmonic functions implies that locally this hypersurface divides the base into chambers. If $Z_{\gamma_1}=kZ_{\gamma_2}$, then $Z_{\gamma_1-k\gamma_2}=0$. In particular, $dZ_{\gamma_1-k\gamma_2}=0$ together with Lemma \ref{34} implies 
  \begin{align*}
\partial \gamma_1-k\partial \gamma_2=0\in H_1(L,\mathbb{Z})\cong \mathbb{Z}^2
  \end{align*}
 Thus, from the exact sequence (\ref{731}) there exists positive integers $k_1=kk_2,k_2$, and $\partial\gamma'\in H_1(L,\mathbb{Z})$, such that we have
 \begin{align*}
 \partial\gamma_i=k_i\partial \gamma' \in H_1(L,\mathbb{Z})), \mbox{ $i=1,2$}.
 \end{align*} and 
 \begin{align*}
     \partial\gamma=\partial\gamma_1+\partial\gamma_2=(k_1+k_2)\partial\gamma'.
 \end{align*}Thus, $\partial \gamma$ is not primitive. To sum up, we have proved the following theorem:
   
   \begin{thm}\label{811} Assume that 
    \begin{enumerate}
          \item the relative class $\gamma\neq 0$, with $\partial \gamma \neq 0\footnote{For the case $\partial \gamma=0$, one also has to consider the situation when a rational curve with a point on $L$ appears as real codimension one boundary of the moduli space}\in H_1(L)$ being primitive.
           
          \item The fibre $L_u$ does not sit over the (real analytic) Zariski closed subset of $B_0$ locally given by 
                \begin{align}
                  W'_{\gamma_1,\gamma_2}=&\{u\in B_0|\mbox{ there are holomorphic discs in relative class } \\ \notag
                  & \mbox{$\gamma_{1,u}$ and $\gamma_{2,u}$ with respect to same complex structure $J_{\vartheta}$, for some $\vartheta\in S^1$}\}. 
                \end{align}Notice that there are only finitely many possible decomposition $\gamma=\gamma_1+\gamma_2$ such that $W'_{\gamma_1,\gamma_2}$ is non-empty by Gromov compactness theorem.
        \end{enumerate}
   Then the moduli space $\mathcal{M}_{k,\gamma}(\mathfrak{X},L_u)$ has no real codimension one boundary.
\end{thm}

 Since the topology of $L_u$ are torus, we will use the trivial spin structure for defining orientations on $\mathcal{M}_{k,\gamma}(\mathfrak{X},L_u)$ \footnote{For general holomorphic Lagrangian $L$ with non-trivial deformation, let $B$ be the moduli space of deformation (as holomorphic Lagrangians) of $L$ in $(X,\Omega)$ which is an analytic variety. Let $B_0=B\backslash \Delta$, where $\Delta$ is the discriminant locus. Since deformation of smooth holomorphic Lagrangians is unobstructed, $B_0$ is a smooth complex manifold. All the argument will be still valid but one need to include the relative spin structures as the decoration of relevant moduli spaces to define the invariants.} For each point $u\in B_0$, it represents a holomorphic Lagrangian $L_u$. Using the techniques in \cite{FOOO}\cite{F1}, the virtual fundamental cycle $[\mathcal{M}_{k,\gamma}(\mathfrak{X},L_u)]^{vir}$ is defined in \cite{L4}. Using the de Rham model introduced in \cite{F1}, we have the following definition of open Gromov-Witten invariants.

\begin{definition}\label{733}Under the same assumption of Theorem \ref{811}.
Let $\phi_1,\cdots,\phi_k \in H^*(L_u)$, the open Gromov-Witten invariants are defined by
 \begin{align}  \label{705}
   \langle \phi_1,\cdots, \phi_k\rangle_{k,u,\gamma}=\int_{[\mathcal{M}_{k,\gamma}(\mathfrak{X},L_u)]^{vir}}\bigwedge_{i=1}^k ev^*\phi_i.
 \end{align} 
\end{definition}
 
 For dimension reason, the expression (\ref{705}) vanishes unless 
       \begin{align*}
        n-2=\sum (\mbox{deg}\phi_i-1). 
       \end{align*}   

\begin{rmk}
   On can actually construct the Kuranishi structures and perturbed multi-sections on $\{\mathcal{M}_{k,\gamma}(\mathfrak{X},L)\}_{k\geq 0}$ with certain compatibility conditions similar to \cite{L4}. In particular, there exists a cyclic filtered $A_{\infty}$
   algebra structure modulo $T^{E_0}$ on $H^*(L\times
   S^1_{\vartheta})$ with $1$ as a strict unit. Moreover, the structure
   is independent of the choice of Kuranishi structures and
   multi-sections chosen, up to pseudo-isotopy of inhomogeneous cyclic
   filtered $A_{\infty}$ algebras.
\end{rmk}
       
The Definition \ref{733} a priori depends on the Lagrangian boundary condition $L_u$. Let $u_0,u_1\in B_0$ and assume that there is a path $p(t):[0,1]\rightarrow B_0$ such that $p(0)=u_0$ and $p(1)=u_1$. Let $\gamma_i\in H_2(X,L_{u_i})$ such that they are parallel transport of each other via the Gauss-Manin connection of the local system $\Gamma$ along the path $p(t)$, thus we will denote them just by $\gamma$. Similiarly, We identify the cohomology $H^*(L_{u_0})$ and $H^*(L_{u_1})$ this way and just denote it by $H^*(L)$. Then the corbordism argument shows that the open Gromov-Witten invariants (\ref{705}) is locally an constant.  
    \begin{thm} \label{708}
      Assume $p(t)\not \in W'_{\gamma}$ for every $t\in [0,1]$, then 
        \begin{align*}
           \langle \phi_1,\cdots, \phi_k\rangle_{k,u_0,\gamma}= \langle \phi_1,\cdots, \phi_k\rangle_{k,u_1,\gamma},
        \end{align*}for $\phi_1,\cdots,\phi_k\in H^*(L)$.
    \end{thm}
When the path $p(t)$ indeed intersects $W_{\gamma}$, the open Gromov-Witten invariants might jump. Therefore, it makes sense not to define the open Gromov-Witten invariants on $W'_{\gamma}$. We will discuss an example of non-trivial jumping of the open Gromov-Witten in the case of elliptic K3 surface in the next section.

In the definition of the invariant $ \langle \phi_1,\cdots, \phi_k\rangle_{k,u,\gamma}$, the family $\mathfrak{X}$ a priori depends on
a choice of Ricci-flat metric.
However, using the similar cobordism argument and Theorem \ref{708}, we have
\begin{cor} \label{216}
Assume $\omega,\omega'$ are Ricci-flat metrics such that the
corresponding invariants $ \langle \phi_1,\cdots, \phi_k\rangle_{k,u,\gamma}$ and
$ \langle \phi_1,\cdots, \phi_k\rangle'_{k,u,\gamma}$ are well-defined. Then
    \begin{align*}
        \langle \phi_1,\cdots, \phi_k\rangle_{k,u,\gamma}= \langle \phi_1,\cdots, \phi_k\rangle'_{k,u,\gamma}.\
    \end{align*}
\end{cor}

\section{Ellitpic K3 Surfaces}
This section we will focus on the case when the hyperK\"ahler manifold $X$ is an elliptic K3 surface with $24$ singular $I_1$-type singular fibres. Since all the moduli space $\mathcal{M}_{0,\gamma}(\mathfrak{X},L)$ has virtual dimension zero, it is natural to understand the virtual count of holomorphic discs with no marked points. 

\begin{definition}
Let $X \rightarrow B$ be an K3 surface with elliptic fibration. For $u\in B_0$ and $\gamma \in H_2(X,L_u)$, the open Gromov-Witten invariants of elliptic K3 surface is defined to be
    \begin{align} \label{707}
     \tilde{\Omega}(\gamma;u):=\int_{[\mathcal{M}_{0,\gamma}(\mathfrak{X},L_u)]^{vir}}1.
    \end{align}
\end{definition}
Similarly as the discussion in the previous section, the invariant $\tilde{\Omega}(\gamma;u)$ is well-defined whenever $u$ does not fall on $W'_{\gamma}$. Unlike the general hyperK\"ahler case, we don't need the primitive and generic assumption to make $\tilde{\Omega}(\gamma;u)$ well-defined for dimension reason. Now we can define the wall of marginal stability as follows:
\begin{definition}
  Let $\gamma \in \Gamma$, the wall of marginal stability associated to $\gamma$ is defined by 
    \begin{align}
      W_{\gamma}:=\bigcup_{\gamma=\gamma_1+\gamma_2}\{u\in W'_{\gamma_1,\gamma_2}| \exists d_1,d_2\in \mathbb{N}, \mbox{ s.t. }\tilde{\Omega}(\frac{\gamma_1}{d_1};u)\tilde{\Omega}(\frac{\gamma_2}{d_2};u)\neq 0\}.
    \end{align}
\end{definition}
Since the central charge $Z$ is holomorphic, $W_{\gamma}$ locally is union of smooth real analytic curves, of real codimension one on $B$. 

\begin{rmk}
  For hyperK\"ahler surfaces, only a real codimension one locus of
  the space of almost complex structures can bound pseudo-holomorphic
  discs with Maslov index zero of a given relative class. A more general way of defining the invariant is to construct a
  $1$-parameter family of (almost) complex structures which is transverse
  to the above real codimension one locus. From Remark \ref{300}, the equator of the twistor line provides such a natural candidate of the
  auxiliary curve. The different choices of the auxiliary curves might give rise to different invariants. The difference is governed
  by the topological intersection of the real codimension one locus and the auxiliary curve, which can be viewed as the
  real analogue of the Noether-Lefschetz numbers in algebraic geometry. We will focus on this viewpoint more in \cite{L6}
\end{rmk}
Using the cobordism argument, we have the following proposition.
\begin{prop} \label{352}
Assume $\partial \gamma\neq 0\in H_1(L_{u_0})$ and
$\tilde{\Omega}(\gamma;u_0)$ is well-defined, then
  $\tilde{\Omega}(\gamma;u)$ is well-defined and locally a constant around
  $u_0$.
\end{prop}
The argument is valid as long as the the path we choose does not hit the wall of marginal stability. In other words, the wall of marginal stability locally divide the base into chambers and $\tilde{\Omega}(\gamma)$ is a constant inside each chamber locally. 

Since as topological spaces,  $\mathcal{M}_{0,\gamma}(\mathfrak{X},L_u)$ and $\mathcal{M}_{0,-\gamma}(\mathfrak{X},L_u)$ are homeomorphic. By checking the orientation, we prove that $\tilde{\Omega}(\gamma)$ satisfies the "reality condition", which is expected for gneralized Donaldson-Thomas invariants in the twistorial construction of hyperK\"ahler metric \cite{GMN} (or see the review in section 6):
\begin{thm}\label{800}\cite{L4} If $\tilde{\Omega}(\gamma;u)$ is well-defined,
then $\tilde{\Omega}(-\gamma;u)$ is also well-defined.
Moreover, we have the reality condition
  \begin{align*}
   \tilde{\Omega}(\gamma;u)=\tilde{\Omega}(-\gamma;u).
  \end{align*}
\end{thm}

We expect that there is a corresponding ``Gopakumar-Vafa type" invariants for this open Gromov-Witten invariant $\tilde{\Omega}(\gamma)$.  Moreover, they are related by a M\"obius inversion type transform.
\begin{conj} \label{222}
  There exists $\{\Omega(\gamma;u)\in \mathbb{Z}| \gamma \in H_2(X,L_u)\}$
  such that
     \begin{equation*}
         \tilde{\Omega}(\gamma;u)=\sum_{d>
         0}\pm d^{-2}\Omega(\gamma/d;u),
     \end{equation*} where the $\pm$ sign depends on the quadratic refinement \cite{GMN}.
\end{conj}

\subsection{Local Model: the Ooguri-Vafa Space} \label{801}
All the definitions and arguments in the previous sections apply to any hyperK\"ahler surfaces (not necessarily compact) whenever the Gromov compactness holds.  The Ooguri-Vafa space is an elliptic fibration over a unit disc such that the only singular fibre is a nodal curve (or an $I_1$-type singluar fibre) over the origin. The singular central fibre breaks the $T^2$-symmetry into only $S^1$-symmetry. Using Gibbons-Hawkings ansatz, Ooguri and Vafa\cite{OV} wrote down Ricci-flat metrics with a $S^1$ symmetry and thus the Ooguri-Vafa space is hyperK\"ahler. From the discussion in section 5, there exists an $S^1$-family of integral affine structures on $B_0$. The central fibre is of $I_1$-type implies that the monodromy of affine structure (see the section 6.1) around the singularity is conjugate to $\bigl(
\begin{smallmatrix}
  1 & 1\\
  0 & 1
\end{smallmatrix} \bigr)$. Thus there exists an unique affine line $l_{\vartheta}$ passing through the singularity in the monodromy invariant direction.
      
Following the maximal principle trick in \cite{A}\cite{C}, if we fix $\vartheta$, the only simple holomorphic discs are in the relative class of Lefschetz thimble $\gamma_e$ and with their boundaries on torus over $l_{\vartheta}$. Topologically, this holomorphic disc is the union of vanishing cycles over $l_{\vartheta}$. This reflects the fact the standard moduli space of holomorphic discs has virtual dimension minus one and thus generic torus fibres would not bound holomorphic discs. However, when the $\vartheta$ goes around $S^1$, the affine line $l_{\vartheta}$ will rotate $2\pi$ and every point on the base will be exactly swept once by $l_{\vartheta}$. In other words, every torus fibre bounds a unique simple holomorphic disc (up to orientation) in Ooguri-Vafa space but with respect to different complex structures. See Figure 1 below. 
\begin{figure}[htb]
\begin{center}
\includegraphics[height=3in,width=6in]{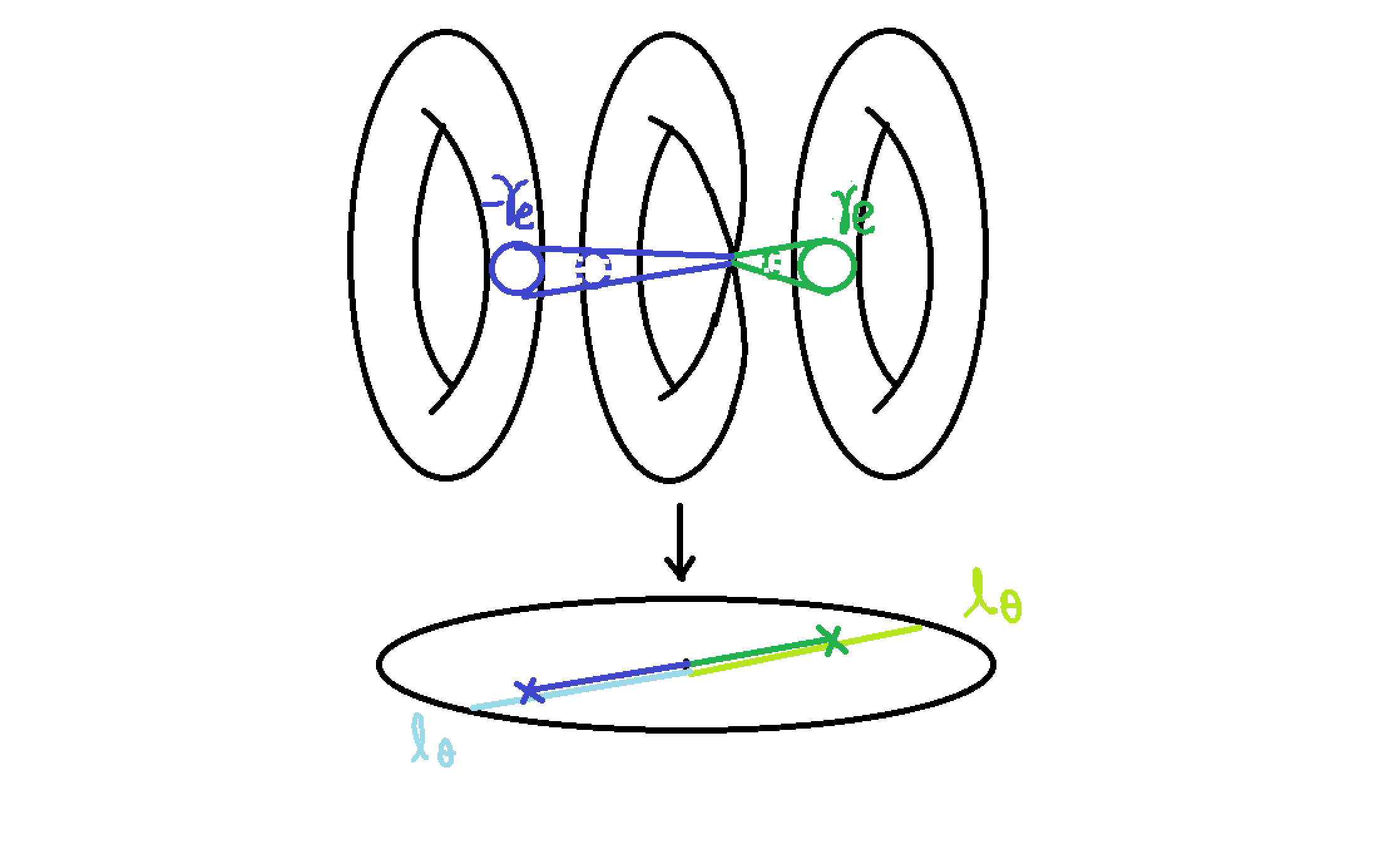}
\caption{The Ooguri-Vafa space and its unique simple holomorphic disc.}
\end{center}
\end{figure}
 
The unique simple holomorphic disc in the Ooguri-Vafa space is Fredholm regular in the $S^1$-family. Thus, it is straightforward to show that $\tilde{\Omega}(\pm\gamma_e)=1$. Moreover, we compute the multiple cover formula in \cite{L4},
   \begin{align*}
         \tilde{\Omega}(\gamma,u)=\begin{cases} \frac{(-1)^{d-1}}{d^2}& \text{, if $\gamma=  d\gamma_e$, $d\in \mathbb{Z}$}\\
                0& \text{, otherwise}.
                                                          \end{cases}
        \end{align*}
   
\begin{rmk}
  Notice that the invariant computed above indeed satisfies Conjecture \ref{222}. Moreover, it suggests that
     \begin{align*}
      \Omega(\gamma,u)=\begin{cases} 1& \text{if $\gamma= \pm\gamma_e,$}\\
             0& \text{otherwise,}
                                                       \end{cases}
     \end{align*}which coincides with the BPS counting for the Ooguri-Vafa space \cite{GMN}.
\end{rmk}        
        
 \subsection{Open Gromov-Witten Invariants from Local}
Before we trying to understand the open Gromov-Witten invariants on K3 surface, we first want to know the existence of holomorphic discs. It is natural to ask if the simple holomorphic disc in Ooguri-Vafa space can actually "live" near the $I_1$-singular fibre of a K3 surface. Let $X_{\vartheta}\rightarrow B$ be a K3 surface with special Lagrangian fibration. Assume $p$ is a singularity of complex affine structure corresponding to an $I_1$-type singular fibre. Similar to the proof of Proposition \ref{600}, there exists two affine rays (with respect to the complex affine coordinate) starting from $p$ such that $Z_{\gamma_e}$ has constant phase. Using the estimate in \cite{GW} and deformation of special Lagrangians with boundaries \cite{B}, we have the following result: 
\begin{thm} \label{60} \cite{L4}
  Let $u$ be a point on the above affine ray starting at the singular point $p$. Assume there is no other singular point of affine structure on the affine segment between $u$ and $p$. Then there exists $\epsilon_0=\epsilon_0(u)>0$ such that there exists an immersed holomorphic disc in the relative class $\gamma_e$ and boundary on $L_u$ in $X_{\vartheta}'$. Here $X_{\vartheta}'$ is the K3 surface with special Lagrangian fibration derived from hyperK\"ahler rotation but using any other different hyperK\"ahler metric $\omega'$ of elliptic fibred K3 surface $X$ such that $\int_{L_u}\omega'<\epsilon_0$. 
\end{thm}

\begin{rmk}
 When $\epsilon_0$ goes to zero, the size of special Lagrangian torus fibres also goes to zero. It is proved that the K3 surface collapses to the base affine manifold \cite{GW}. This is exactly the picture of large complex structure limit from the point of view of Strominger-Yau-Zaslow conjecture \cite{GW} \cite{KS4}\cite{SYZ}.
\end{rmk}
 
The first example we can compute the new invariant $\tilde{\Omega}$ is the following one and also its multiple cover formula: 
\begin{thm} \label{43} \cite{L4}
   Let $\gamma_e$ be the relative class of Lefschetz thimble around an $I_1$-type singular fibre, then given any $d_0\in \mathbb{N}$, there exists a non-empty neighborhood $\mathcal{U}$ of the singularity such that for each $u\in \mathcal{U}$, we have 
  \begin{align*}
   \tilde{\Omega}(d\gamma_{e}; u)=\frac{(-1)^{d-1}}{d^2}, \mbox{ for every integer $d$, $|d|\leq d_0$}.
  \end{align*}
   Moreover, for $u$ close enough to the singularity,
   $\pm\gamma_e$ are the only classes achieve minimum energy with $\tilde{\Omega}(\gamma)\neq 0$.
\end{thm}   
 The proof is essentially solving a complex Monge-Ampere equation to construct a family of hyperK\"ahler structures connecting the one on the Ooguri-Vafa space and the one on a neighborhood of $I_1$-type singular fibre in K3 surface $X$. Then we can use the cobordism argument similar in Theorem \ref{708}.

\begin{rmk}
The latter part of Theorem \ref{43} implies the affine ray corresponding to $\gamma_e$ from Proposition \ref{600} indeed corresponds to the initial ray in Gross-Seibert program \cite{GS1}. Moreover, the generation function 
    \begin{align*}
      X_{\gamma_e}=\sum_{d\in \mathbb{N}}d\tilde{\Omega}(\gamma_e)x^{d\gamma_e}=\log{(1+x^{\gamma_e})}
    \end{align*}exactly coincides with the wall-crossing factor (or the slab function) of the initial ray corresponding to $\gamma_e$.
\end{rmk}

The invariants from other relative classes are generated by wall-crossing formula inductively with respect to the energy which we will discuss a bit in the following section.

\subsection{Wall-Crossing Phenomenon}
Below we will demonstrate a non-trivial example of wall-crossing
phenomenon of invariants $\tilde{\Omega}(\gamma;u)$ on elliptic K3
surfaces $X$.

\begin{ex} \label{702}\cite{L4}
  Assume there are two initial rays emanating from two $I_1$-type singularities of phase $\vartheta_0$ intersect transversally
  at $p\in B_0$. From Theorem \ref{60}, there are two initial holomorphic
  discs of relative classes $\gamma_1,\gamma_2$ corresponding to the initial rays which are Fredholm. Then $p$ sits on a wall of marginal stability $W_{\gamma_1,\gamma_2}$. Moreover, the local model provided in the proof of
  Theorem \ref{60} indicates that they intersect transversally in $L_p$ when the K3 surface is close enough to the large complex limit. From automatic
  transversality of K3 surfaces, these two discs cannot be smoothed out inside
  $L_p$. To prove that these two discs will smooth when changing the Lagrangian boundary condition, first pick two point $p_1,p_2$ near $p$ but on the
  different sides of the wall of marginal stability
  $W_{\gamma_1,\gamma_2}$. Let $\psi:(-\epsilon, 1+\epsilon)$ be a path on $B_0$ such that $\psi(0)=p_1$, $\psi(1)=p_2$ and intersects $W_{\gamma_1,\gamma_2}$
  once transversally at $p$. Recall that $\mathcal{X}$ is the total space of twistor space of $X$ with two fibres with elliptic fibration deleted. Then $L_u\times S^1_\vartheta$ is a totally real torus in
  $\mathcal{X}$. Now consider an complex manifold $\mathfrak{X}\times \mathbb{C}$
  with a totally real submanifold
    \begin{align*}
       \mathcal{L}=\bigcup_t (L_{\psi(t)}\times S^1_{\vartheta}).
    \end{align*}
  By our assumption, there are two regular holomorphic discs in
  $\mathcal{X}$ with boundaries in
  $L_p\times\{\vartheta_0\}\subseteq \mathcal{L}$ of relative classes again we denoted by $\gamma_1$, $\gamma_2$. The tangent of
  evaluation maps for both discs are two dimensional and
  transversal. By Taubes gluing construction \cite{FOOO}, these two discs can be smoothed out
into simple regular discs in $\mathcal{L}$ and the union of initial
holomorhpic discs is indeed the codimension one of the boundary of
the usual moduli space of holomorphic discs
$\mathcal{M}_{0,\gamma_1+\gamma_2}(\mathfrak{X}\times \mathbb{C},\mathcal{L})$. By maximal principle twice, each of the
holomorphic disc falls in
$\mathcal{M}_{\gamma_1+\gamma_2}(X_{{\vartheta}_0},L_{\psi(t)})$ for some
$t$. In particular,
   \begin{align*}
     \mathcal{M}_{1,\gamma_1}(\mathfrak{X},L_p)\times_{L\times
     S^1_{\vartheta}}\mathcal{M}_{1,\gamma_2}(\mathfrak{X},L_p)\subseteq
     \mathcal{M}_{\gamma_1+\gamma_2}(\mathfrak{X},\{L_t\})
   \end{align*} as codimension one boundary. One can computes that the non-trivial wall-crossing phenomenon indeed occurs:
   \begin{align*}
      \Delta\tilde{\Omega}(\gamma_1+\gamma_2)
      &=\pm\langle\gamma_1,\gamma_2\rangle
      \tilde{\Omega}(\gamma_1,p)\tilde{\Omega}(\gamma_2,p)=\pm \langle \gamma_1,\gamma_2\rangle \neq 0.
   \end{align*}
\end{ex}

In general, we expect the following wall-crossing formula for $\tilde{\Omega}(\gamma)$, which is equivalent to Kontsevich-Soibelman wall-crossing formula \cite{KS2}.
 \begin{conj}
  \label{63} Assume $\gamma$ is a primitive charge. When $u$ cross a wall consisting of relative classes $\gamma_i$, $i=1,\cdots,n$, then one has the
               following wall-crossing formula for $\tilde{\Omega}$:
      \begin{align} \label{37}
         \Delta \tilde{\Omega}(d\gamma)=\sum_{\mathbf{w}:\sum|\mathbf{w_i}|\gamma_i=d\gamma}\frac{N^{trop}(\mathbf{w})}{|Aut(\mathbf{w})|} \bigg(\prod_{1 \leq i \leq n, 1 \leq j \leq l_i} \tilde{\Omega}(w_{ij}\gamma_i)
         \bigg),
      \end{align}where
      $\mathbf{w}=(\mathbf{w}_1,\cdots,\mathbf{w}_n)$,
      $\mathbf{w}_i=(w_{i1},\cdots,w_{il_i})\in \mathbb{Z}^{l_i}_{\geq
      0}$, and $|\mathbf{w}_i|=\sum^{l_i}_{k=1}w_{ik}$. The factor $N^{trop}(\boldsymbol{w})$ is a counting of tropical discs with integer value defined in \cite{GPS}\cite{L4}.
 \end{conj}
\begin{rmk}
Notice that the wall-crossing phenomenon in Example \ref{702} satisfies the Conjecture \ref{63}.
\end{rmk}

\section{Tropical Geometry of K3 Surfaces and the Corresponding Theorem}
In this section, we want to study the tropical geometry of K3 surfaces and a corresponding theorem between holomorphic discs and tropical discs.

Tropical geometry raises naturally from the point of view of modified Strominger-Yau-Zaslow conjecture \cite{GW}\cite{KS4}\cite{SYZ}. Naively, the special Lagrangian fibration in a Calabi-Yau manifold collapses to the base affine manifold when the Calabi-Yau manifold goes to the large complex limit point. The projection of holomorphic curves are amoebas and degenerate to some $1$-skeletons at the limit, which are called "tropical curves" in modern terminology. This idea has been carried out for toric varieties by Mikhalkin \cite{M2} and Lagrangian bundles with no monodromy by Parker \cite{P}. Here we are going to discuss the case where special Lagrangian fibration comes from K3 surfaces, which admits singular fibres and consider holomorphic discs instead of holomorphic curves. 

Let $B$ be the base of a holomorphic Lagrangian fibration of a hyperK\"ahler manifold $X$ then we have an $S^1$-family of integral affine structures on $B_0$. Indeed, for any $\vartheta \in S^1$ and lifting $\gamma_i$ of the generators of $\Gamma_g$, the functions $f_i=Re(e^{-i\vartheta}Z_{\gamma_i})$ give the local  affine coordinates with transition functions in $Sp(2,\mathbb{Z})\ltimes \mathbb{R}^2$. We will denote $B$ together with this integral affine structure by $B_{\vartheta}$. In particular, neither the choice of K\"ahler class $[\omega]$ of the hyperK\"ahler manifold nor the real scaling of the holomorphic $(2,0)$-form $\Omega$ change the affine straight lines on the base affine manifold. One advantage of introducing the affine structure is allowing us to discuss tropical geometry on $B$ \cite{L4}.

\begin{prop}\label{600}
Locally the set of special Lagrangian torus fibres bounding holomorphic
discs of a same relative class in $X_{\vartheta}$ all fall above an affine hyperplane on the base affine manifold $B_{\vartheta}$.
\end{prop}
\begin{proof}
   Assume $\{L_t\}_{t\in (-\epsilon,\epsilon)}$ are a family of special Lagrangian torus fibres bound holomorphic discs in relative class $\gamma_t\in H_2(\underline{X},L_t)$ in $X_{\vartheta}$. Then 
     \begin{align*}
      \int_{\gamma_t}\Omega_{\vartheta}=\int_{\gamma_t}\omega-i\mbox{Re}(e^{-i\vartheta}\Omega)=0. 
     \end{align*}
  In particular, $L_t$ are confined by the equation
     \begin{align*}
       f_{\gamma}:=\mbox{Re}(e^{-i\vartheta}Z_{\gamma})=0,
     \end{align*}which all sit above an affine hyperplane. 
\end{proof}
\begin{rmk}
  From the proof of Proposition \ref{600}, the prescribed affine line is special in the sense that the corresponding central charge $Z_{\gamma}$ has constant phase along $\{f_{\gamma}=0\}$. We will call them special affine lines with respect to phase $\vartheta$.
\end{rmk}
\begin{rmk}
  In dimension two, the affine structure described in Proposition \ref{600} is usually called the complex affine structure of the special Lagrangian in the context of mirror symmetry. 
\end{rmk}   
  
  One notice that using the integral affine structure, the following map
   \begin{align*}
               (&\Gamma_g)_u \longrightarrow \hspace{7mm} T^*_uB \\
                  &\partial \gamma \mapsto   \big(v \in T_uB \mapsto \int_{\partial\gamma}\iota_{\tilde{v}}\mbox{Im}(e^{-i\vartheta}\Omega) \big)
   \end{align*}is injective. Here the notation is same in Proposition \ref{34}. In particular, the above map induces an integral structure on $T_u^*B_{\vartheta}$ and thus on $T_uB_{\vartheta}$. The following is a temporary definition of tropical discs proposed in \cite{G3}.
     
\begin{definition} \label{732}
  A tropical disc on $B_{\vartheta}$ is a $3$-tuple $(\phi,G,w)$ where $G$ is a rooted connected tree with a root $x$. We denote the set of vertices and edges by $G^{[0]}$ and $G^{[1]}$ respectively, with a weight function $w:G^{[1]}\rightarrow \mathbb{N}$. And $\phi:G\rightarrow B$ is a continuous map such that 
  \begin{enumerate}
    \item For each $e\in G^{[1]}$, $\phi|_e$ is an affine segment on $B_{\vartheta}$.  
    \item For the root $x$, $\phi(x)\in B_0$.
    \item For each $v\in G^{[0]}$, $v\neq x$ and $\mbox{val}(v)=1$, we have $\phi(v)\in \Delta$. Moreover, if $\phi(v)$ corresponds to an $I_n$-type singular fibre, then the image of edge adjacent to $v$ is in the monodromy invariant direction. 
    \item For each $v\in G^{[0]}$, $\mbox{val}(v)\geq 1$, we have the following assumption: 
    
      (balancing condition) Each outgoing tangent at $u$ along the image of each edge adjacent to $v$ is rational with respect to the above integral structure on $T_{\phi(v)}B$. Denote the outgoing primitive tangent vectors by $v_i$, then 
          \begin{align*}
           \sum_i w_i v_i=0.
          \end{align*}
  \end{enumerate}
\end{definition}
From now on, we will assume $X$ is an elliptic K3 surface.
Assume $\gamma_u\in H_2(X,L_u)$ can be represented as a holomorphic disc
with boundary on $L_u$ such that
$\tilde{\Omega}(\gamma;u)\neq 0$. Without lose of
generality, we may assume $\int_{\gamma_u}\Omega \in
\mathbb{R}_+$. There is an affine half line $l$ emanating from $y$
on the base such that $\int_{\gamma_t}\Omega$ is a
decreasing function of $t\in l$, where $\gamma_t$ is the
parallel transport of $\gamma_u$ along $l$. Since
$\int_{ \gamma_{t}}\Omega$ is an affine function along $l$,
it has no lower bound. There is some point $u'\in l$ such that
$\int_{\gamma_{u'}}\Omega=0$. Thus, there are two cases:
  \begin{enumerate}
    \item If $\tilde{\Omega}(\gamma;u)=\tilde{\Omega}(\gamma;u')$ then $L_{u'}$ is a singular fibre by the gradient estimate of harmonic maps.  In particular,
    if $L_{u'}$ is an $I_n$-type singular fibre then $\gamma_{u'}$ is represented by multiple cover of the unique area
    minimizing holomorphic disc and $\gamma_u$ is the parallel of $\gamma_{u'}$ along
    $l$.
    \item If $\tilde{\Omega}(\gamma;u)\neq \tilde{\Omega}(\gamma;u')$, then there exists
    $\gamma_{n,u'}$ in the same phase with $\gamma_{u'}$ such that
    $\sum_n\gamma_{n,u'}=\gamma_{u'}$ and $\tilde{\Omega}(\gamma_n;u')\neq 0$.  Then we replace $\gamma$ by each $\gamma_{i,u'}$ and repeat the same
    processes. The procedure will stop at finite time because of
    Gromov compactness theorem. By induction, every holomorphic
    disc with nontrivial invariant give rise to a tropical disc,
    which is formed by the union of the affine segments on the base.
    Indeed, the assumption 3.(a) in Definition \ref{732} is implied the fact that each vertex of above graph falls on $W_{\gamma'}$ for some $\gamma'$. The balancing conditions are guaranteed by the
    conservation of charges $\sum_n\gamma_{n,u'}=\gamma_{u'}$ at
    each vertex $u'$. Using the normal construction, one can associate the tropical disc $\phi$ a relative class $[\phi]$. It is not hard to show that $[\phi]=\gamma$.
  \end{enumerate} To sum up, we proved the following theorem by this
  attractor flow mechanism \cite{DM} of holomorphic discs (See also Figure 2).

\begin{thm} \label{47} \cite{L4}
  Let $X$ be an elliptic K3 surface (with singular fibres not necessarily of $I_1$-type). For every relative class $\gamma\in H_2(X,L_u)$, $[\partial \gamma] \neq 0\in H_1(L_u)$ and $\tilde{\Omega}(\gamma;u)\neq 0$, there is a corresponding
  tropical disc $\phi$. Assume the singular fibres are all of $I_1$-type then $[\phi]=\gamma$ and the symplectic area of the holomorphic disc is just the total affine length of the corresponding tropical disc.
\end{thm}

\begin{figure}
\begin{center}
\includegraphics[height=3in,width=6in]{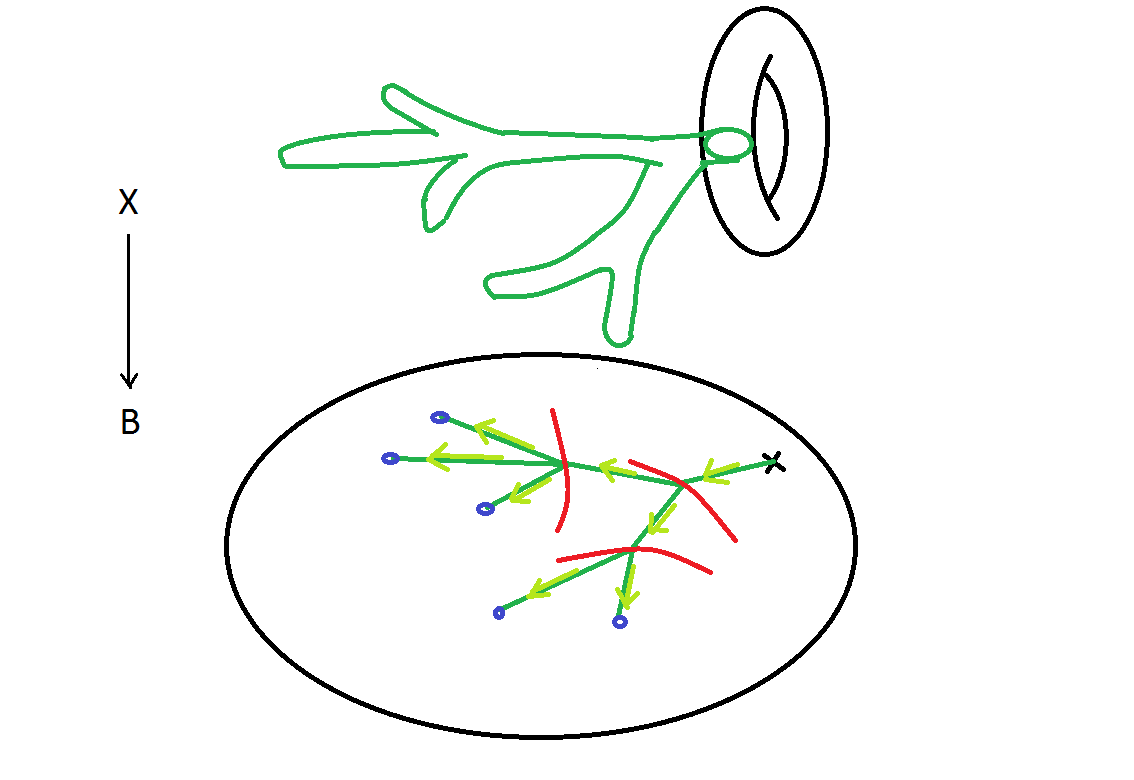}
\caption{Holomorphic discs with non-trivial invariant and its corresponding tropical disc. The red curves are walls of marginal stability and blue little circles are singularity of the affine structure}
\end{center}
\end{figure}

Theorem \ref{47} also helps to understand the topology of holomorphic discs in elliptic K3 surfaces.
\begin{cor} All the holomorphic discs with non-trivial boundary class and with non-trivial open Gromov-Witten invariant in $X$ topologically are coming from scattering" (gluing) of discs coming from singularities.
\end{cor}

Conversely, we can ask if each tropical discs on K3 surface has a lifting holomorphic disc. This involves more detail of analysis of geometry near different kinds of singular fibres. The Conjecture \ref{63} implies this converse statement for the case that all singular fibres are of $I_1$-type. The study of Conjecture \ref{63} will leave for future work. 

\section{Twistorial Construction of HyperK\"ahler Metric}
  The well-known Calabi conjecture solved by Yau guarantees that given
   a compact Calabi-Yau manifold, there exists a unique Ricci-flat
   K\"ahler metric in each prescribed K\"ahler class \cite{Y1} in 1978. After the  existence, an important question to ask is how to write down the
   explicit expression of the metric. The celebrated Strominger-Yau-Zaslow conjecture \cite{SYZ} suggested that
   Calabi-Yau manifolds will admit a special Lagrangian fibration around
   large complex limits and the mirror will be given by the dual
   fibration. It is a folklore that the Ricci-flat metrics near large
   complex limits are approximated by semi-flat metrics with instanton
   correction related to the holomorphic discs with boundaries on
   special Lagrangian fibres \cite{F3}. The first part is done for K3
   surfaces:  in \cite{GVY} the semi-flat metric is wrote down for the special Lagrangian fibration. Later, Gross and Wilson \cite{GW} proved that for elliptic K3 surfaces around large complex limits, the Ricci-flat metrics are approximated by the semi-flat metrics gluing with Ooguri-Vafa metrics. However, the instanton corrections are not included.
  
  Although the semi-flat metric approximates the true Ricci-flat metric near the large complex limit point, the curvature of the semi-flat metric blows up near the singular fibres and thus cannot be extended to the whole K3 (or a general) hyperK\"ahler manifold $X$. To remedy this defect, one has to introduce the "quantum corrections". From the result of \cite{HKLR}, the explicit expression of hyperK\"ahler metric of a hyperK\"ahler manifold can be achieved from holomorphic symplectic $2$-forms with respect to all complex structures parametrized by the twistor line. The idea is to glue pieces of flat space with standard holomorphic symplectic $2$-form via certain symplectormorphisms, which are determined by the generalized Donaldson-Thomas invariants. These invariants are locally some integer-valued functions depending on the charge (thus depending on the base). There are so-called walls of marginal stability separate the base of abelian fibration into chambers locally. The invariants are constants inside the chamber while might jump when across the wall. The jump of the invariants cannot be arbitrary but governed by the Kontsevich-Soibelman wall-crossing formula \cite{KS2}, which suggests the compatibility of the gluing flat pieces with a global holomorphic symplectic $2$-form. The Kontsevich-Soibelman wall-crossing formula is then interpreted as the smoothness of the holomorphic symplectic $2$-form. Therefore, one can construct the holomorphic symplectic $2$-forms $\Omega_{\zeta}$ for $\zeta \in \mathbb{C}^*$ inside the twistor line $\mathbb{P}^1$. Using the reality condition (see the Theorem \ref{800}), \cite{GMN} argues that the Cauchy-Riemann equation can indeed extend over whole twistor line $\mathbb{P}^1$.

In \cite{GMN2}, Gaiotto-Moore-Neitzke found out the correct notion of generalized Donaldson-Thomas invariants and carried out the above recipe for Hitchin moduli spaces. It still remains open for the construction of hyperK\"ahler metric for general abelian fibred hyperK\"ahler manifolds, especially compact ones such as K3 surfaces. Follow the the recipe in \cite{GMN}, the key step is to find the corresponding generalized Donaldson-Thomas invariants satisfying various properties. Here, we introduce the reduced open Gromov-Witten invariants which are conjectured to satisfy the Kontsevich-Soibelman wall-crossing formula and serve this purpose.  

\begin{rmk}
  However, there are still essential difficulties to carry out the
  twistorial construction of hyperK\"ahler metric on compact hyperK\"ahler manifolds. First of all, the
  metric is only constructed outside of singular torus fibres and it
  is hard to check if the metric can be extended. Secondly, for
  compact hyperK\"ahler manifolds, the generalized Donaldson-Thomas invariants might has growth rate faster than
  exponential growth and make the formal solution of
  relevant Riemann-Hilbert problem diverge.
\end{rmk}

\appendix
\flushbottom
\include{AppendixA}
\flushbottom

Department of Mathematics, Stanford University\\
 E-mail address: yslin221@stanford.edu 


\bib{B}{article}{
  author={Butscher, Adrian},
  title={Deformations of minimal {L}agrangian submanifolds with boundary},
  journal={Proc. Amer. Math. Soc.},
  volume={131},
  date={2003},
  number={6},
  pages={1953--1964},
}

\bib{A}{article}{
  author={Auroux, Denis},
  title={Mirror symmetry and $T$-duality in the complement of an anticanonical divisor},
  journal={J. G\"okova Geom. Topol. GGT},
  volume={1},
  date={2007},
  pages={51--91},
}

\bib{C}{article}{
  author={Chan, Kwokwai},
  title={The {O}oguri-{V}afa metric, holomorphic discs and wall-crossing},
  journal={Math. Res. Lett.},
  volume={17},
  number={3},
  date={2010},
  pages={401--414},
}

\bib{DM}{article}{
  author={Denef, Frederik},
  author={Moore, Gregory W.},
  title={Split states, entropy enigmas, holes and halos},
  journal={J. High Energy Phys.},
  date={2011},
  number={11},
  pages={129, i, 152},
}

\bib{F1}{article}{
  author={Fukaya, Kenji},
  title={Cyclic symmetry and adic convergence in {L}agrangian {F}loer theory},
  journal={Kyoto J. Math.},
  volume={50},
  number={3},
  date={2010},
  pages={521--590},
}

\bib{F3}{article}{
  author={Fukaya, Kenji},
  title={Multivalued {M}orse theory, asymptotic analysis and mirror symmetry},
  booktitle={Graphs and patterns in mathematics and theoretical physics},
  series={Proc. Sympos. Pure Math.},
  volume={73},
  pages={205--278},
  publisher={Amer. Math. Soc.},
  address={Providence, RI},
  year={2005},
}

\bib{FOOO}{book}{
  author={Fukaya, Kenji},
  author={ Oh, Yong-Geun},
  author={Ohta, Hiroshi},
  author={Ono, Kaoru},
  title={Lagrangian intersection {F}loer theory: anomaly and obstruction. {P}art {II}},
  series={AMS/IP Studies in Advanced Mathematics},
  volume={46},
  publisher={American Mathematical Society},
  address={Providence, RI},
  year={2009},
  pages={i--xii and 397--805},
}

\bib{G}{article}{
  author={Gromov, Mark},
  title={Pseudo holomorphic curves in symplectic manifolds},
  journal={Inventiones Mathematicae},
  volume={82},
  year={1985},
  pages={307--347},
}

\bib{G3}{article}{
  author={Gross, Mark},
  title={The Strominger-Yau-Zaslow conjecture: from torus fibrations to degenerations},
  conference={ title={Algebraic geometry---Seattle 2005. Part 1}, },
  book={ series={Proc. Sympos. Pure Math.}, volume={80}, publisher={Amer. Math. Soc.}, place={Providence, RI}, },
  date={2009},
  pages={149--192},
}

\bib{GPS}{article}{
  author={Gross, Mark},
  author={Pandharipande, Rahul},
  author={Siebert, Bernd},
  title={The tropical vertex},
  journal={Duke Math. J.},
  fjournal={Duke Mathematical Journal},
  volume={153},
  year={2010},
  number={2},
  pages={297--362},
}

\bib{GMN}{article}{
  author={Gaiotto, Davide},
  author={Moore, Gregory W.},
  author={Neitzke, Andrew},
  title={Four-dimensional wall-crossing via three-dimensional field theory},
  journal={Comm. Math. Phys.},
  fjournal={Communications in Mathematical Physics},
  volume={299},
  year={2010},
  number={1},
  pages={163--224},
}

\bib{GMN2}{article}{
  author={Gaiotto, Davide},
  author={Moore, Gregory W.},
  author={Neitzke, Andrew},
  title={Wall-crossing, Hitchin systems, and the WKB approximation},
  journal={Adv. Math.},
  volume={234},
  date={2013},
  pages={239--403},
}

\bib{GS1}{article}{
  author={Gross, Mark},
  author={Siebert, Bernd},
  title={From real affine geometry to complex geometry},
  journal={Ann. of Math. (2)},
  fjournal={Annals of Mathematics. Second Series},
  volume={174},
  year={2011},
  number={3},
  pages={1301--1428},
}

\bib{GW}{article}{
  author={Gross, Mark},
  author={Wilson, P. M. H.},
  title={Large complex structure limits of {$K3$} surfaces},
  journal={J. Differential Geom.},
  fjournal={Journal of Differential Geometry},
  volume={55},
  year={2000},
  number={3},
  pages={475--546},
}

\bib{GVY}{article}{
  author={Greene, Brian R.},
  author={Shapere, Alfred},
  author={Vafa, Cumrun},
  author={Yau, Shing-Tung},
  title={Stringy cosmic strings and noncompact {C}alabi-{Y}au manifolds},
  journal={Nuclear Phys. B},
  fjournal={Nuclear Physics. B},
  volume={337},
  year={1990},
  number={1},
  pages={1--36},
}

\bib{HKLR}{article}{
  author={Hitchin, N. J.},
  author={Karlhede, A.},
  author={Lindstr{\"o}m, U.},
  author={Ro{\v {c}}ek, M.},
  title={Hyper-{K}\"ahler metrics and supersymmetry},
  journal={Comm. Math. Phys.},
  fjournal={Communications in Mathematical Physics},
  volume={108},
  year={1987},
  number={4},
  pages={535--589},
}

\bib{KS1}{article}{
  author={Kontsevich, Maxim},
  author={Soibelman, Yan},
  title={Affine structures and non-{A}rchimedean analytic spaces},
  booktitle={The unity of mathematics},
  series={Progr. Math.},
  volume={244},
  pages={321--385},
  publisher={Birkh\"auser Boston},
  address={Boston, MA},
  year={2006},
}

\bib{KS2}{article}{
  author={Kontsevich, Maxim},
  author={Soibelman, Yan},
  title={Stability structures, motivic Donaldson-Thomas invariants and cluster transformations},
  journal={preprint 2008, math. AG/ 0811.2435v1.},
}

\bib{KS4}{article}{
  author={Kontsevich, Maxim},
  author={Soibelman, Yan},
  title={Homological mirror symmetry and torus fibrations},
  conference={ title={Symplectic geometry and mirror symmetry}, address={Seoul}, date={2000}, },
  book={ publisher={World Sci. Publ., River Edge, NJ}, },
  date={2001},
  pages={203--263},
}

\bib{L}{article}{
  author={Lee, Junho},
  title={Family {G}romov-{W}itten invariants for {K}\"ahler surfaces},
  journal={Duke Math. J.},
  fjournal={Duke Mathematical Journal},
  volume={123},
  year={2004},
  number={1},
  pages={209--233},
}

\bib{L2}{article}{
  author={C.-C. Melissa Liu},
  title={Moduli of J-Holomorphic curves with Lagrangian Boundary conditions and Open Gromov-Witten Invariants for an $S^1$-Equivariant Pair},
  journal={preprint 2002, math.SG/0210257v2.},
}

\bib{L4}{article}{
  author={Lin, Yu-Shen},
  title={Open Gromov-Witten Invariants on Elliptic K3 Surfaces and Wall-Crossing},
  journal={in preparation.},
}

\bib{L6}{article}{
  author={Lin, Yu-Shen},
  title={Real Noether-Lefschetz Theory and Open Gromov-Witten Invariants},
  journal={in preparation.},
}

\bib{L5}{article}{
  author={Liu, Ai-Ko},
  title={Cosmic string, {H}arvey-{M}oore conjecture and family {S}eiberg-{W}itten Theory},
  journal={preprint 2004, math.AG/0409038.},
}

\bib{OV}{article}{
  author={Ooguri, Hirosi},
  author={Vafa, Cumrun},
  title={Summing up {D}irichlet instantons},
  journal={Phys. Rev. Lett.},
  fjournal={Physical Review Letters},
  volume={77},
  year={1996},
  number={16},
  pages={3296--3298},
}

\bib{M2}{article}{
  author={Mikhalkin, Grigory},
  title={Enumerative tropical algebraic geometry in $\Bbb R^2$},
  journal={J. Amer. Math. Soc.},
  volume={18},
  date={2005},
  number={2},
  pages={313--377},
}

\bib{MP}{article}{
  author={Pandharipande, Rual},
  author={Maulik, Davesh},
  title={Gromov-Witten theory and Noether-Lefschetz theory},
  journal={preprint 2007, math.AG/0705.1653},
}

\bib{P}{article}{
  author={Parker, Brett},
  title={Holomorphic Curves in Lagrangian Fibrations},
  journal={Thesis (Ph.D.)?Stanford University. 2005.},
}

\bib{S2}{article}{
  author={J. Solomon},
  title={Intersection Theory on the Moduli Space of Holomorphic Curves with Lagrangian Boundary Conditions},
  journal={preprint 2006, math.SG/0606429v1.},
}

\bib{SYZ}{article}{
  author={Strominger, Andrew},
  author={Yau, Shing-Tung},
  author={Zaslow, Eric},
  title={Mirror symmetry is {$T$}-duality},
  journal={Nuclear Phys. B},
  fjournal={Nuclear Physics. B},
  volume={479},
  year={1996},
  number={1-2},
  pages={243--259},
}

\bib{Y1}{article}{
  author={Yau, Shing-Tung},
  title={On the Ricci curvature of a compact K\"ahler manifold and the complex Monge--Amp\`ere equation. I},
  journal={Comm. Pure Appl. Math.},
  volume={31},
  date={1978},
  number={3},
  pages={339--411},
}



\begin{bibdiv}
\begin{biblist}
\bibselect{file001}
\end{biblist}
\end{bibdiv}
\end{document}